\documentclass[12pt]{amsart}     

\usepackage{amssymb}
\usepackage{amsmath}
\usepackage{amsthm}

\usepackage{esint}

\DeclareMathOperator{\essinf}{essinf}

\usepackage[T2A]{fontenc}
\usepackage[cp1251]{inputenc}
\newtheorem {theorem} {Theorem}
\newtheorem {proposition} [theorem]{Proposition}
\newtheorem {corollary} [theorem]{Corollary}
\newtheorem {lemma}  [theorem]{Lemma}
\theoremstyle{definition}
\newtheorem{df}{Definition}

\theoremstyle{remark}
\newtheorem{rem}{Remark}

\newcommand{\rd}{\mathbb{R}^d}

\begin{document}

\title{On characterization of   weakly porous sets via dyadic coverings}



\author{Andrei V.~Vasin }

\address{St.~Petersburg Department
of Steklov Mathematical Institute, Fontanka 27, St.~Petersburg
191023, Russia}
\address{Admiral Makarov State University of Maritime and Inland Shipping,
Dvinskaya st.~5/7, St.~Petersburg 198035, Russia}

\email{andrejvasin@gmail.com}

\thanks{This research was supported by the Russian Science Foundation (grant No.~23-11-00171).}



\begin{abstract}
 We consider the class of weakly porous sets in Euclidean spaces. As our
 main goal,   we  give a precise  characterization in terms of dyadic
covering  of these sets. Also, we obtain the  Carleson embedding inequality  for porous sets.
\end{abstract}

\maketitle

\section{Introduction and main results}\label{s_int}
\subsection{Background}
A set $E$ of the unit circle $\mathbb{T} $ is called porous if there exists a constant $C$ such that for any arc  $I\subset\mathbb{T}$ there
is  a subarc $M(I)\subset I$ containing no points of $E$ and satisfying
\[
|M(I)|> C | I|.
\]
  Dyn'kin proved the following uniform  entropy characterization of the porous set \cite{D}:
\begin{equation}\label{e_pri}
\frac{1}{|I|}\int_{I}\log\frac{1}{dist (x, E)}dx\ < \log\frac{1}{|I|}+C,
\end{equation}
where $C$ is not depending on $I$. There is  a discrete form of (\ref{e_pri})
\begin{equation}\label{e_prd}
  \sum_{J \subset I,\;J\cap E=\emptyset}|J| \log \frac{1}{|J|}\leq |I|\left( \log \frac{1}{|I|} + C\right)
  \end{equation}
 with  $C$ independent of $I$.

The porosity condition  and estimate (\ref{e_pri})  appear naturally
in the free interpolation problems for different classes of analytic functions
smooth up to the boundary of the unit disc $\mathbb{D} \subset \mathbb{R}^2$ (see \cite{D}).

 Studying  the  weak embedding property  of singular inner functions,  Borichev, Nicolau and Thomas  \cite[Lemma 7]{BNT}  found  the Carleson type characterization of porosity: a set $E\subset\mathbb{T}$ is porous if and only if there exists $C$ such that  for  any arc $I \subset \mathbb{T}$ one has
\begin{equation}\label{e_prdya}
\sum_{J \cap E \neq \emptyset,\; J \in \mathcal{D}(I)}|J|\leq C |I|,
\end{equation}
 where $\mathcal{D}(I)$ is a dyadic decomposition  of $I$.

Thus, in  (\ref{e_prd}) the description of a porous set involves arcs free from points of the set, while in (\ref{e_prdya})  arcs intersecting the set are used.

Our main goal in this paper is to extend    property (\ref{e_prdya})
to the higher dimensional case, and to find explicit relations between (\ref{e_prd}) and (\ref{e_prdya}) as well. Note that   extension of  the Dyn'kin inequality to the higher dimensional case is known to be related to the Aikawa
condition  (see \cite{Dy} and references therein).
Also, a further extension of Borichev-Nicolau-Thomas characterization together with the dyadic technique 
 (see
 Meyer-Coifman  \cite[p. 59, Lemma 5]{MC}) leads to the dyadic Carleson embedding inequality  for the porous sets.

We deal  with  certain generalization, and consider weakly porous sets in $\mathbb{R}^d$.
 These sets    on the unit circle $\mathbb{T}$ appeared   in   the study  of  the properties of the Fuchsian groups \cite{V}, where
it   was proved  that  the  function
 $\log dist^{-1}(x, E)$
 belongs to the space $BLO(\mathbb{T})$ if and only if $E$ is weakly porous.
  Anderson, Lehrb\"{a}ck,  Mudarra and  V\"{a}h\"{a}kangas \cite{ALMV} extended  the definition of
   weakly porous sets to $\mathbb{R}^d$. Also, they find  when  the distance function
 $dist^{-\alpha}(x, E)$
 belongs to the Muckenhoupt class $A_1$.
A precise quantitative characterization in terms of the so-called
Muckenhoupt exponent of $E$ is also given in \cite{ALMV}.

 \subsection{Dyadic decomposition}
 Throughout this paper, we consider $\rd$ equipped with the Euclidean distance and the
$d$-dimensional Lebesgue  measure.    The Lebesgue  measure of $E$ is denoted by $|E|$, and if $x \in \rd$, then $ dist(x;E)$
denotes the distance from $x$ to the set $E$.

We only consider cubes which are half-open and have sides parallel to the
coordinate axes. That is, a cube in $\rd$ is a set of the form
\[Q = [a_1; b_1)\times \dots   \times [a_d; b_d)\]
with side-length $\ell(Q) = b_1-a_1 = \dots = b_d-a_d$.
The dyadic decomposition of a cube $R\subset \rd$ is
\[ \mathcal{D}(R)= \bigcup_{j\geq 0} \mathcal{D}_j(R)\]
where each $\mathcal{D}_j(R)$ consists of the $2^{jd}$  pairwise disjoint (half-open) cubes $Q$, with side length
$\ell(Q) = 2^{-j}\ell(R)$, such that
\[R= \bigcup_{Q \in \mathcal{D}_j(R)} Q\]
for every $j =0,1, \dots $. The cubes in  $\mathcal{D}(R)$ are called dyadic cubes (with respect to $R$).
Let $j >0$  and $Q \in \mathcal{D}_j(R)$. Then there exists a unique dyadic cube $\pi Q \in \mathcal{D}_{j-1}(R)$
satisfying $Q \subset \pi Q$, and $|\pi Q|=2^d|Q|$.
\begin{df}
  Let $E$ be a nonempty set.
  For each cube $R\subset\mathbb{R}^d$  let
\[\mathcal{D}(R,E)=\{Q \in \mathcal{D}(R): \:  Q \cap E \neq \emptyset\}\]
be a family of all dyadic cubes intersecting $E$, and
let
\[\mathcal{F}(R,E)= \{Q' \in \mathcal{D}(R): Q'\neq R,\; \:  Q' \cap E = \emptyset, \; \pi Q'\cap E \neq \emptyset\}\]
 be a family of all maximal   pairwise disjoint dyadic cubes not intersecting $E$.

Denote by  $M(R)$  one of the largest dyadic cubes $Q\in \mathcal{D}(R)$ not containing points of $E$.
\end{df}

We start with a  key result about  changing the order of summation that may be of independent interest.

\begin{proposition}
 \label{p_id}
  Let $E \in \rd$ be a  set  such that closure  of $E$ has zero Lebesgue measure.  Then
 \begin{equation}\label{e_id}
 \sum_{Q \in \mathcal{D}(R,E)}|Q|
 = \sum_ {Q'\in \mathcal{F}(R,E) }|Q'|\log \frac{\ell(R)}{\ell(Q')}.
\end{equation}
\end{proposition}
\begin{rem}
  The integral version  of Proposition \ref{p_id} is better known. Given a cube $R$, $\ell(R)<1$, and a closed set $E$ with $|E|=0$, one has the following formula of integration:
\[\int_{0}^{1} \frac{\mu(\lambda)}{\lambda}d\lambda = \int_{R}\log \frac{1}{dist(x,E)}dx,\]
where $\mu(\lambda)=|\{x \in R: dist(x,E)< \lambda\}| $ is the  distribution function.
\end{rem}
 \begin{rem}
 In  this paper, we refer  to  the binary logarithm  $\log x=\log_2 x$.
  \end{rem}
 The sums in (\ref{e_id}) are bounded for  a set  $E$  of finite entropy over $R$:
  \[\int_{R}\log\frac{1}{dist(x, E)}dx < \infty.\]
  It follows from the next result for sets of general type.
\begin{proposition}\label{p_en}
  Let $E \in \rd$ be a  set  such that the closure  of $E$ has zero Lebesgue measure.  Then
 for each cube $R \subset \rd$
 \begin{equation}\label{e_en}
  \left | \sum_ {Q'\in \mathcal{F}(R,E) }|Q'|\log \frac{1}{\ell(Q')} - \int_{R}\log\frac{1}{dist (x, E)}dx \right | < C |R|,
\end{equation}
 and
\begin{equation}\label{e_inf}
  \left | \log \frac{1}{\ell(M(R))} - \inf_R\log\frac{1}{dist (x, E)} \right | < C
\end{equation}
with the constant $C=C(d)$.
\end{proposition}

The original Meyer-Coifman lemma  concerns to  the Carleson embedding inequality for the dyadic covering of $\rd$.
\begin{lemma} [{\cite[p. 59, Lemma 5]{MC}}]
   Let $\mathcal{D'}\subset\mathcal{D}(\rd)$ be a  family of dyadic cubes in $\rd$ and  let $p_Q$, $Q\in  \mathcal{D'}$, be the sequence of positive numbers such that
  \[\sum_{Q\in \mathcal{D'},\: Q\subset R} p_Q\leq C|R|\]
 for each   $R \in  \mathcal{D}(\rd)$ with $C$ independent of $R$. Then for every sequence $a_Q\geq 0$, $Q \in  \mathcal{D'}$, we have
 \begin{equation} \label{e_cm}
 \sum_{Q \in \mathcal{D'}} a_Q p_Q  \leq
C\int_{\rd} \:\sup_{Q \ni x,\,Q \in \mathcal{D'}} a_Q \: dx.
 \end{equation}
\end{lemma}
Here we prove.
\begin{proposition}\label{p_carl}
  Let $E \subset\rd$ be a  set  such that closure  of $E$ has zero Lebesgue measure.
   For each cube $R \subset \rd$ let $\{a_Q\}_{\mathcal{D}(R,E)}$ be a sequence of non-negative scalars. Then there is a  constant $C$ which depends only on $d$ and $E$ such that

\[\int_R \:\sum_{Q \in \mathcal{D}(R,E)} a_Q \chi_Q \:dx \leq
\int_R \:\sup_{Q \in \mathcal{D}(R,E)} a_Q \chi_Q  \:\left(\log \frac{\ell(R)}{dist(x,E)}+C\right)dx.\]
Here $\chi_Q$ is the characteristic function of a cube $Q$.
\end{proposition}
Observe that the estimate above is sharp in a sense.
\subsection{Weakly porous sets}
\begin{df}
  Let $E$ be a nonempty set.
  A set $E\subset\mathbb{R}^d$ is called weakly porous, if there  are constants $0< c, \, \delta <1$ such that for each cube $R\subset\mathbb{R}^d$
  there exists a  finite family $\mathcal{F}(R,E, \delta)  \subset \mathcal{D}(R)$ of pairwise disjoint  dyadic cubes $Q\in \mathcal{D}(R)$  such that $Q\cap E=\emptyset$ and $|Q|\geq \delta |M(R)|$, and
   \begin{equation}\label{e_wpde}
\sum_{Q\in \mathcal{F}(R,E, \delta) }|Q|\geq c |R|.
\end{equation}
\end{df}
Instead of dyadic cubes, also general subcubes of $R$ could be used in the definition of
weak porosity. However, the dyadic formulation \cite{ALMV} is convenient from the point of view of our
proofs.

The following properties  are easy to verify using the definition of weak porosity:
\begin{enumerate}
 \item [W1.]If $E\subset \rd$ is  porous, then $E$ is weakly porous.
  \item [W2.] $E\subset \rd$ is  weakly porous if and only if the closure $E$ is weakly porous.
  \item [W3.]If $E\subset \rd$ is weakly porous, then $|E|=0$. This is a consequence of the Lebesgue
differentiation theorem.
  \item [W4.] Weak porous set  is nowhere dense.
\end{enumerate}

By the John-Nirenberg  theorem (see \cite[p.145]{S}),  there is  a condition of the weak porosity equivalent to \cite[Theorem 1.1]{ALMV}.

\begin{proposition}[{\cite {ALMV, V}}]\label{p_blo}
  A set $E \subset \rd$ is weakly porous if and only if the function
 $\log dist ^{-1}(x, E)$
 belongs to the space $BLO(\rd)$. This means  that
 there exists $C$ such that for each cube $R\subset\rd$ it holds
 \[
\frac{1}{|R|}\int_{R}\log\frac{1}{dist(x, E)}dx\leq \essinf_R \log\frac{1}{dist(x, E)}+C.
\]
\end{proposition}

\subsection{Main results for weakly porous sets.}
 As application of Proposition \ref{p_id}, we give the  following characterization of  weakly porous sets by means of dyadic coverings.

On  the one side we prove the  discrete version of the BLO condition from Proposition \ref{p_blo}, which involves the dyadic covering of complement of a weakly porous set.

On the other  side we obtain  the   description  by means of the  dyadic covering of  the weakly porous sets themselves.

\begin{theorem}\label{t_main}
A set  $E\subset\mathbb{R}^d$ is weakly porous if and only if there is a constant $C$ such that for each cube
$R\subset\mathbb{R}^d$ any of the following properties hold.

\begin{enumerate}
 \item[(i)]  For the family $\mathcal{F}(R,E)$ of dyadic cubes it holds
  \[
 \sum_{Q' \in \mathcal{F}(R,E)}|Q'| \log \frac{1}{\ell(Q')}\leq \left( \log \frac{1}{\ell(M(R))} + C\right) |R|.
\]

  \item[(ii)]  For the family $\mathcal{D}(R,E)$ of dyadic cubes it holds
 \[
\sum_{Q \in \mathcal{D}(R,E)}|Q|\leq \left(\log \frac{\ell (R)}{\ell (M(R))}+C \right)|R|.
\]

 \item[(iii)]  Let
  \[\mathcal{D}_0 (R, E)=\{Q \in \mathcal{D}(R,E):\;   |Q|\leq |M(R)| \},\]
then
  \[
\sum_{Q \in\mathcal{D}_0 (R,E) }|Q|\leq C |R|.
\]

\end{enumerate}
\end{theorem}
\begin{rem}
  Note that in all properties (i)--(iii) we have the same constant $C$.
\end{rem}
\subsection{Main results for porous sets.}
For the porous sets, when  $|R|\leq C |M(R)|$  with  $C$ independent of  $R$,  we have an extension to $\rd$  of  the   results from \cite{BNT, D} and
application of  \cite[p. 59]{MC}.

\begin{theorem} \label{cor3}
A set  $E\subset\mathbb{R}^d$ is porous if and only if there is a constant $C$ such that for each cube
$R\subset\mathbb{R}^d$ any of the following properties hold.
\begin{enumerate}
   \item[(i)] The Dyn'kin inequality: for each cube  $R\subset\mathbb{R}^d$, one has
 \[
\sum_{Q' \in \mathcal{F}(R,E)}|Q'|\log \frac{1}{\ell(Q')}\leq |R|\left( \log \frac{1}{\ell(R)} + C\right).
\]

   \item[(ii)] The Borichev-Nicolau-Thomas characterization: for each cube  $R\subset\mathbb{R}^d$ one has
    \[
\sum_{Q \in \mathcal{D}(R,E)}|Q|\leq C|R|,
\]
which is the  Carleson measure property
\item[(iii)]
The  Carleson embedding inequality: let $1\leq p< \infty$  and $\{a_Q\}_{\mathcal{D}(R,E)}$ be a sequence of non-negative scalars. Then there is the constant $C$ which depends only on $d$, $p$ and $E$ such that
\[\left \| \left(\sum_{Q \in \mathcal{D}(R,E)} a^p_Q \chi_Q \right )^{1/p}\right\|_p \leq C
\left \|\sup_{Q \in \mathcal{D}(R,E)} a_Q \chi_Q \right\|_p. \]
Here $\chi_Q$ is the characteristic function of a cube $Q$, and
\[\|f\|_p=\left(\int_{\rd}|f(x)|^pdx\right)^{1/p}.\]

\end{enumerate}
\end{theorem}
Observe that (iii) with maximal-function techniques implies the following inequality, which is related to   \cite[Theorem 2.10]{IV}. 
\begin{corollary} \label{cor4}
Let a set  $E\subset\mathbb{R}^d$ be porous,  $1\leq p< \infty$, $1\leq q\leq \infty$  and let $\{a_Q\}_{\mathcal{D}(R,E)}$ be a sequence of non-negative scalars. Then there is the constant $C$ which depends only on $d, \,p$ and $E$ such that

\[\left \| \sum_{Q \in \mathcal{D}(R,E)} a_Q \chi_Q  \right\|_p \leq C
\left \|\left (\sum_{Q \in \mathcal{D}(R,E)} a^q_Q \chi_Q \right )^{1/q} \right\|_p .\]
\end{corollary}
 \section{Proofs}

\subsection{Proof of Proposition \ref{p_id}}

  The underlying ideas are in principle similar to
those in  \cite{BNT}, but the higher dimensional case requires several nontrivial modifications.

 Cubes of the family $\mathcal{F}(R,E)$ are pairwise disjoint and
 \[\bigcup_{Q' \in \mathcal{F}(R,E)} Q'  \supseteq  R \backslash \overline{E}\]
   for any cube $R$.  Since
 $|\overline{E}|=0$, it holds that
 \[
 \sum_ {Q'\in \mathcal{F}(R,E)}|Q'|=|R|.
 \]
So, for   the left hand side of  (\ref{e_id}) we have, by changing the summation  order,
\[
\aligned
\sum_{Q \in \mathcal{D}(R,E)} |Q|
&= \sum_ {Q\in \mathcal{D}(R,E) }\; \sum_{ Q' \in \mathcal{F}(R,E)} |Q\cap Q'| \\
&= \sum_ {Q'\in \mathcal{F}(R,E) }\; \sum_{ Q \in \mathcal{D}(R,E)} |Q\cap Q'| .
\endaligned
\]
If $Q, Q' \in \mathcal{D}(R)$ and $Q'\cap Q \neq \emptyset$, then  $Q\subset Q'$ or $Q'\subset Q$.
Since $Q\cap E\neq \emptyset$ and $Q'\cap E=\emptyset$, it holds $Q'\subset Q$ and so
$|Q'\cap Q|=|Q'|$.
 Therefore,
\[
\aligned
\sum_ {Q\in \mathcal{D}(R,E)}|Q|
&=\sum_ {Q'\in \mathcal{F}(R,E)} \quad\sum_{ Q \in \mathcal{D}(R,E),\: Q\supset Q'} |Q'| \\
&=\sum_ {Q'\in \mathcal{F}(R,E)} |Q'| \quad\sum_{ Q \in \mathcal{D}(R,E),\: Q\supset Q'} 1 \\
&= \sum_ {Q'\in \mathcal{F}(R,E) }|Q'| \;\mathcal{N} (Q \in \mathcal{D}(R,E),\: Q\supset Q'),
\endaligned
\]
where
$\mathcal{N} (Q \in \mathcal{D}(R,E),\: Q\supset Q')$ is a number of cubes from $\mathcal{D}(E,R) $ containing $Q'$. This equals to  a  number  of embedded dyadic cubes  in the sequence
\[
Q\subset \pi Q\subset \dots \pi^n Q,
\]
where the largest cube is $\pi^n Q = R$ and  the smallest cube is $Q=\pi Q' $. Indeed, for each cube  $Q' \in \mathcal{F}(R,E)$,   it holds  $\pi Q' \cap E \neq \emptyset$, and hence $Q=\pi Q' \in \mathcal{D}(R,E) $. Therefore,

 \[
 \mathcal{N} (Q \in \mathcal{D}(E,R), Q\supset Q') =\log \frac{\ell(R)}{\ell( Q')},
 \]
as required.
\subsection{Proof of Proposition \ref{p_en}} We need an elementary computational lemma.
\begin{lemma}\label{l_ele}
Let $Q \subset \rd$ be a cube with side-length $\ell(Q)$ and Lebesgue measure $|Q|$, and the boundary $\partial Q$. Then
\[ \frac{1}{|Q|}\int_{Q} \log \frac{1}{dist (x, \partial Q)}dx= \log \frac{1}{\ell(Q)} +C\]
with the  constant $C=C(d)$.
\end {lemma}
\begin{proof} By scaling and translation invariance one can reduce  the lemma  to the case of the fixed cube $Q_0=[0;1)\times \dots \times[0;1)$. Then one can apply Proposition \ref{p_blo} with $R=Q_0$ and $E=\partial Q_0$ to obtain the result.

Alternatively, by elementary calculations, we have
\[
\aligned
 \frac{1}{|Q|}\int_{Q} \log \frac{1}{dist (x, \partial Q)}dx
 &=\frac{1}{|Q|}\int_{Q_0} \log \frac{1}{\ell\: dist (x, \partial Q_0)}\ell^d d x\\
 &= \log \frac{1}{\ell}+ \int_{Q_0} \log \frac{1}{dist (x, \partial Q_0)}dx\\
 &= \log \frac{1}{\ell}+C.
\endaligned
\]
\end{proof}

Now, we prove (\ref{e_en}).
Observe that
  for each cube $Q'\in \mathcal{F}(R,E) $ and $x \in Q'$ one has
 \[dist (x,E) \leq 2 \sqrt{d}\ell(Q'),\]
 hence
\begin{equation}\label{e_11}
  \log \frac{1}{\ell(Q')}\leq  \log \frac{2 \sqrt{d}}{dist (x,E)}.
\end{equation}
  Therefore, since $|\overline{E}|=0$, it holds
 \[
 \aligned
 \sum_{Q' \in \mathcal{F}(R,E)}|Q'| \log \frac{1}{\ell(Q')} &\leq \sum_{Q' \in \mathcal{F}(R,E)}\int_{Q'} \log \frac{2 \sqrt{d}}{dist (x,E)}dx\\
 & \leq \int_{R} \log \frac{2 \sqrt{d}}{dist (x,E)}dx\\
  & = \int_{R} \log \frac{1}{dist (x,E)}dx + C |R|
 \endaligned
 \]
with $C= \log 2 \sqrt{d}$.

 Conversely,   we have
\[
 \aligned
 \int_{R} \log \frac{1}{dist (x,E)}dx  & =\sum_{Q' \in \mathcal{F}(R,E)}\int_{Q'} \log \frac{1}{dist (x,E)}dx\\
 & \leq \sum_{Q' \in \mathcal{F}(R,E)}\int_{Q'} \log \frac{1}{dist (x, \partial Q')}dx,
 \endaligned
 \]
where $\partial Q'$ is the boundary of a cube $Q'$. Lemma \ref{l_ele} completes the proof of (\ref{e_en}).

To prove (\ref{e_inf}), we argue as in  the proof of (\ref{e_en}). So, by (\ref{e_11}), for each cube $Q'\in \mathcal{F}(R,E) $ and $x \in Q'$, one has
\[
 \aligned
\log \frac{1}{\ell(M(R))} &\leq  \log \frac{1}{\ell(Q')}\\
& \leq  \log \frac{2 \sqrt{d}}{dist (x,E)},
\endaligned
 \]
which follows
\[
 \log \frac{1}{\ell(M(R))} \leq \inf_R \log \frac{2 \sqrt{d}}{ dist(x,E)}.
 \]
On the other hand,
\[
\aligned
 \sup_{R} dist (x,  E) &\geq \sup_{M(R)} dist (x,E)\\
& \geq \frac{\ell(M(R))}{2},
\endaligned
 \]
therefore,
\[
 \log \frac{1}{2 \sqrt{d} \ell(M(R))} \leq \inf_R\log \frac{1}{dist(x,E)}\leq \log \frac{2}{\ell(M(R))},
 \]
which easily gets  (\ref{e_inf}).
\subsection{Proof of Proposition \ref{p_carl}}
First, we observe the sharpness of Proposition \ref{p_carl}. Put $a_Q=1$ if $Q \in  \mathcal{D}(R,E)$, then by  Proposition \ref{p_en}
and  by Proposition \ref{p_id},
 we obtain for these $a_Q$ the inverse inequality
\[
\aligned
\int_R \:\sup_{Q \in \mathcal{D}(R,E)}  \chi_Q  \:\log \frac{\ell(R)}{dist(x,E)}dx
&=\int_R \:\log \frac{\ell(R)}{dist(x,E)}dx\\
&\leq \sum_{Q' \in \mathcal{F}(R,E)}|Q'|\: (\log \frac{ \ell(R)}{\ell(Q')}+C)\\
&=  \sum_{Q \in \mathcal{D}(R,E)} |Q| +C|R|\\
&=  \int_R \;\sum_{Q \in \mathcal{D}(R,E)}  \chi_Q \:dx +C|R|
 \endaligned
 \]
 with  $C=C(E)$.
 Thus, we have
\[ \int_R \:\sup_{Q \in \mathcal{D}(R,E)}  \chi_Q  \:\left(\log \frac{\ell(R)}{dist(x,E)}-C\right) dx \leq \int_R \;\sum_{Q \in \mathcal{D}(R,E)}  \chi_Q \:dx.\]

In the forward direction  Proposition \ref{p_id} together with
dyadic arguments \cite[p. 59, Lemma 5]{MC} imply  Proposition \ref{p_carl}.
Indeed,
let $\chi_{(0< t< a_Q)}$  be the characteristic function, then
\[
 \aligned
 \int_R \:\sum_{Q \in \mathcal{D}(R,E)} a_Q \chi_Q \:dx
 &= \sum_{Q \in \mathcal{D}(R,E)} a_Q |Q| \\
 & = \int_0^\infty \sum_{Q \in \mathcal{D}(R,E)} |Q| \chi_{(0< t< a_Q)} dt.
 \endaligned
 \]
 The set
 \[\Omega_t= \{x:\:\sup_{Q \in \mathcal{D}(R,E)} a_Q \chi_Q >t   \}\]
 can be expressed as the union of all the cubes $Q$ such that $a_Q>t$. Let  $Q_k$ denote the maximal dyadic cube contained in $\Omega_t$,
then
$\Omega_t$ is the union of all the maximal cubes $Q_k$, and for $t>0$ we get
\[
 \aligned
  \sum_{Q \in \mathcal{D}(R,E)} |Q| \chi_{(0< t< a_Q)}
  & \leq \sum_{Q \in \Omega_t} |Q|\\
  & = \sum_k \sum_{Q \in Q_k} |Q|.
 \endaligned
 \]
To this end, we followed \cite[p. 59, Lemma 5]{MC}. Now, using Proposition \ref{p_id},  the property  $|Q_k|\leq|R|$, and  Proposition \ref{p_en}, we get
 \[
  \aligned
  \sum_{Q \in Q_k}|Q|
  &= \sum_ {Q'\in \mathcal{F}(Q_k,E) }|Q'| \log \frac{\ell(Q_k)}{\ell(Q')}\\
  &\leq \sum_ {Q'\in \mathcal{F}(Q_k,E) }|Q'| \log \frac{\ell(R)}{\ell(Q')}\\
  &\leq \int_{Q_k}\:\left( \log \frac{\ell(R)}{dist(x,E)}+C\right)dx.
    \endaligned
  \]
 Therefore, we obtain the estimate on the double sum
  \[
 \aligned
    \sum_k \sum_{Q \in Q_k} |Q|
    & \leq\sum_k  \int_{Q_k}\:\left( \log \frac{\ell(R)}{dist(x,E)}+C\right)dx\\
    & =\int_{\Omega_t}\:\left( \log \frac{\ell(R)}{dist(x,E)}+C\right) dx.
 \endaligned
 \]
 We finish by observing
 \[
   \int_0^\infty dt \int_{\Omega_t}\: \left( \log \frac{\ell(R)}{dist(x,E)}+C\right)dx =
 \int_R \:\sup_{Q \in \mathcal{D}(R,E)} a_Q \chi_Q  \:\left( \log \frac{\ell(R)}{dist(x,E)}+C\right)dx.
  \]

\subsection{Proof of Theorem \ref{t_main}} By  Proposition   \ref{p_en} and  Proposition   \ref{p_blo},  the weak porosity of $E$  is equivalent to (i).

(i) $\Leftrightarrow $ (ii) easily follows from  Proposition \ref{p_id}.

(i) $\Leftrightarrow $ (iii) Arguing as  in   Proposition \ref{p_id}, replace the family $\mathcal{D}(R, E)$ by the  family
 $\mathcal{D}_0 (R, E)$. It holds
\[
\aligned
\sum_{Q \in \mathcal{D}_0(R,E)} |Q|
& = \sum_ {Q\in \mathcal{D}_0 (R,E) } \quad \sum_{ Q' \in \mathcal{F}(R,E)} |Q\cap Q'| \\
& = \sum_ {Q'\in \mathcal{F}(R,E) } \quad \sum_{ Q \in \mathcal{D}_0(R,E)} |Q\cap Q'| \\
&=\sum_ {Q'\in \mathcal{F}(R,E)}\quad \sum_{ Q \in \mathcal{D}_0(R,E),\; Q\supset Q'} |Q'| \\
&= \sum_ {Q'\in \mathcal{F}(R,E) }|Q'| \,\mathcal{N} (Q \in \mathcal{D}_0(R,E), \;Q\supset Q'),
\endaligned
\]
where $\mathcal{N} (Q \in \mathcal{D}_0(R,E), \;Q\supset Q')$  is the number of cubes from $\mathcal{D}_0(R,E)$ in the  sequence
$ Q\subset \pi Q\subset \dots \pi^n Q$. Here $\ell(\pi^n Q) =\ell( M(R))$ and $Q=\pi Q'$, that is,
\[
\mathcal{N} (Q \in \mathcal{D}_0(R,E), \;Q\supset Q') = \log \frac{\ell(M(R))}{\ell(Q')}.
\]
Therefore,
\[
\sum_{Q \in \mathcal{D}_0(R,E)} |Q|= \sum_ {Q'\in \mathcal{F}(R,E) }|Q'| \log \frac{\ell(M(R))}{\ell(Q')} .
\]
Since both  sums above are bounded by $C|R|$  at the same time,  one gets (i) $\Leftrightarrow $ (iii), which
  completes the proof of Theorem \ref{t_main}.

\subsection{Proof of Theorem \ref{cor3}. } By Proposition \ref{p_id}, the porosity condition equivalent  (i) and (ii).

 We  prove  that (ii) and (iii) are equivalent.
 Testing inequality (iii) with $p = 1$  and the choice $a_Q=1$  if $Q \in  \mathcal{D}(R,E)$, implies
 (ii).

 Conversely,
if $E$ is porous, then property (ii)  together  with \cite[p. 59, Lemma 5]{MC} already cited imply (iii).
Indeed, after redefinitions we may assume that $p=1$.
Arguing and using the notions as  in   Proposition \ref{p_carl}, we write
\[
 \aligned
 \int_R \:\sum_{Q \in \mathcal{D}(R,E)} a_Q \chi_Q \:dx
  & = \int_0^\infty \sum_{Q \in \mathcal{D}(R,E)} |Q| \chi_{(0< t< a_Q)} dt\\
   & \leq \int_0^\infty   \sum_{Q \in \Omega_t} |Q| dt\\
  & =  \int_0^\infty \sum_k \sum_{Q \in Q_k} |Q|dt,
 \endaligned
 \]
where
 $\Omega_t= \{x:\:\sup_{Q \in \mathcal{D}(R,E)} a_Q \chi_Q >t   \}$
 is expressed as the union of all the maximal cubes $Q_k$  contained in $\Omega_t$.
Now, we use (ii)  to estimate the double sum
\[\sum_k \sum_{Q \in Q_k} |Q|\leq C \sum_k |Q_k|= C|\Omega_t|,\]
 and  observing
 \[
 \aligned
 \int_0^\infty  |\Omega_t|dt
&= \int_R \:\sup_{Q \in \mathcal{D}(R,E)} a_Q \chi_Q  \:dx,
  \endaligned
  \]
 we complete the proof.
\subsection{Proof of Corollary \ref{cor4}.}
If $p=1$ this is obvious consequence of (iii),  Theorem \ref{cor3}.

If $p>1$, it suffices to prove for $q=\infty$. We follow     \cite[Theorem 2.10]{IV} and prove  by duality
\[
  I:=\int_R \left|\sum_{Q \in \mathcal{D}(R,E)} a_Q \chi_Q (x) \psi(x)\right| dx \leq C
  \left \|\sup_{Q \in \mathcal{D}(R,E)} a_Q \chi_Q  \right\|_p 
\]
for every $\psi \in C^\infty (R)$ such that $\|\psi\|_{L^{p'}(R)}=1$, $1/p+1/{p'}=1$. Let $\psi$ be such a test function. Then 
\[
 \aligned
  I 
  &\leq \sum_{Q \in \mathcal{D}(R,E)} a_Q \int_Q |\psi|(x)|dx\\
  &=\sum_{Q \in \mathcal{D}(R,E)} a_Q  |\psi|_Q |Q|\\
  \endaligned  
\]
where $|\psi|_Q$ denotes the mean value over $Q$. Now we use the proof of Theorem \ref{cor3} replacing 
$a_Q$ by $a_Q  |\psi|_Q$ and obtain 
\[
 \aligned
I
& \leq C \int_R \:\sup_{Q \in \mathcal{D}(R,E)} a_Q |\psi|_Q\chi_Q (x) \:dx\\
 &\leq C \int_R \:\sup_{Q \in \mathcal{D}(R,E)} a_Q \chi_Q (x)  M\psi(x) \:dx,
  \endaligned 
 \]
 where $M\psi(x)$ stands for the non-centered Hardy-Littlewood maximal function  over cubes.
  We finish applying the H\"{o}lder inequality and the boundedness of the maximal operator in $L^{p'}$.


\bibliographystyle{amsplain}

\end{document}